\documentclass[11pt]{amsart}
\usepackage[english]{babel}
\usepackage{amssymb}
\usepackage{amsmath}
\usepackage{srcltx}

\newcommand{\ignore}[1]{}

\newtheorem{dummy}{Dummy}

\newtheorem{theorem}[dummy]{Theorem}
\newtheorem{proposition}[dummy]{Proposition}
\newtheorem{corollary}[dummy]{Corollary}

\theoremstyle{definition}

\newtheorem{remark}[dummy]{Remark}

\newcommand{\wh}{\widehat}



\author{S. Pumpl\"un}
\email{susanne.pumpluen@nottingham.ac.uk}
\address{School of Mathematical Sciences\\
University of Nottingham\\
University Park\\
Nottingham NG7 2RD\\
United Kingdom
 }
\date{19.3.2018}

\begin{document}

\title[Diagonal forms of higher degree]{Diagonal forms of higher degree over function fields of $p$-adic curves}

\maketitle

\begin{abstract}
We investigate diagonal forms of degree $d$
over the function field $F$ of a smooth projective $p$-adic curve: if a form is isotropic over the
completion of $F$ with respect to each discrete valuation of $F$,
then it is isotropic over certain fields $F_U$, $F_P$ and $F_p$.
These fields appear naturally when applying the methodology of patching; $F$ is the inverse limit of the finite inverse
system of fields $\{F_U,F_P,F_p\}$.
Our observations complement some known bounds on the higher $u$-invariant of diagonal forms of degree $d$.

We only consider diagonal forms of degree $d$ over fields of characteristic not dividing $d!$.
\end{abstract}

\section*{Introduction}

 The fact that Springer's Theorem holds  for diagonal forms of higher degree over fields of characteristic not dividing $d!$ \cite{Mo}
guarantees that on occasion diagonal forms of higher degree defined over function fields behave similarly to quadratic forms.
For a survey on the behaviour of (diagonal) forms of higher degree
in general the reader is referred to \cite{Pu}.

In this note we consider diagonal forms of degree $d$ over function fields $F = K(X)$ where
 $X$ is a smooth, projective, geometrically integral curve
over $K$ and $K$ is the fraction field of a complete discrete valuation ring
 with a residue field $k$ of characteristic not dividing $d!$.
Let $v$ be a rank one discrete valuation of $F$, and $F_v$ the completion of $F$ with respect to $v$.
It was shown by Colliot-Th\'el\`ene,  Parimala and Suresh  \cite[Theorem 3.1]{CT-P-S}
that a quadratic form which is isotropic over  $F_v$
for each $v$ is already isotropic over $F$,  using the methodology
 of patching developed by Habater and Hartmann \cite{H-H}, i.e. viewing $F$ as  the inverse limit of a finite inverse
system of certain fields $\{F_U,F_P,F_p\}$.

Given a nondegenerate diagonal form $\varphi$  over $F$ of degree $d>2$ and dimension $>2$,
 it is not clear, however, whether the isotropy of  $\varphi$ over $F_v$
for each $v$ implies that $\varphi$ is isotropic.

Our main result proves that the isotropy of a nondegenerate diagonal form $\varphi$ over $F_v$
for each $v$ implies that  at least over the field extensions $F_U$, $F_P$ and $F_p$ of $F$, $\varphi$ is isotropic as well
(Theorem \ref{thm:21}). These fields depend on the choice of the form
$\varphi=\langle a_1,\dots,a_n\rangle$, more precisely on the choice of the regular proper model
 ${\mathcal X}$ (over the complete discrete valuation ring  $A$) of the curve $X$ over $K$, which depends on $\varphi$:
  ${\mathcal X}$  is selected such  that there exists a reduced divisor
$D$ with strict normal crossings, which contains both the support of the divisor of all the entries $a_i$,
$1\leq i\leq n$,
and the components of the special fibre of $X/A$. Since nondegenerate diagonal forms of degree $d\geq 3$
have finite automorphism groups  \cite[p.~137]{H}, we are not able to apply
 \cite[Theorem 3.7]{H-H-K} to conclude that the isotropy of $\varphi$ over the $F_U$'s and $F_P$'s implies
 that $\varphi$ is also isotropic over $F$, however. This is only possible for $d=2$.

After collecting the terminology and some basic results in Section \ref{sec:prel}, in particular defining
diagonal $u$-invariants of degree $d$ over $k$,
we consider diagonal forms of higher degree over valued fields in Section \ref{sec:valued}
and then study diagonal forms of higher degree over function fields of $p$-adic curves
 using some of the ideas of \cite{CT-P-S}   in Section \ref{sec:Parimala}.
Recall that a $p$-adic field is a finite field extension of $\mathbb{Q}_p.$

 As a consequence of Springer's Theorem for diagonal forms, any diagonal form of degree $d$ and dimension $>d^3+1$  over
a function field in one variable $F=K(t)$, where $K$ is a $p$-adic field with residue field $k$,
${\rm char} (k) \nmid d!$,
 is isotropic over $F_v$ for every discrete valuation $v$
with residue field either a function field in one variable over $k$ or a finite extension of $K$.
Moreover, it is isotropic over $F_U$ for each reduced, irreducible component $U\subset Y$ of the
complement of $S$ in the special fibre $Y= {\mathcal X}\times_Ak$  of $X/A$,
 and  isotropic over  $F_P$ for each $P\in S$ (Corollary \ref{cor:5}), and thus isotropic over
$F_p$ for each $p=(U,P)$. Here, $S$ is the inverse image
 under a finite $A$-morphism $f: {\mathcal X}\to\mathbb{P}_A^1$ of the point at infinity of the special fibre $\mathbb{P}_k^1$.

%
%
\section{Preliminaries} \label{sec:prel}

Let $k$ be a field such that ${\rm char}( k)$ does not divide $d!$.

\subsection{Forms of higher degree}

Let $V$ be a finite-dimensional vector space over $k$ of dimension $n$.
A $d$-{\it linear form} over $k$ is a $k$-multilinear map $\theta : V \times
\dots \times V \to k$ ($d$-copies)  which is {\it symmetric}, i.e. $\theta (v_1,
\dots, v_d)$ is invariant under all permutations of its variables.
A {\it form of degree $d$} over $k$  (and of dimension $n$) is a map $\varphi:V\to k$
 such that $\varphi(a v)=a^d\varphi(v)$ for all $a\in k$, $v\in V$ and such that the map $\theta : V \times
\dots \times V \to k$ defined by
 $$\theta(v_1,\dots,v_d)=\frac{1}{d!}\sum_{1\leq i_1< \dots<i_l\leq d}(-1)^{d-l}\varphi(v_{i_1}+ \dots +v_{i_l})$$
($1\leq l\leq d$) is a $d$-linear form over $k$.
By fixing a basis $\{e_1,\dots,e_n\}$ of $V$, any form $\varphi$ of degree $d$ can be viewed as a homogeneous polynomial
of degree $d$ in $n={\rm dim}\, V$ variables $x_1,\dots,x_n$
 via $\varphi(x_1,\dots,x_n)=\varphi(x_1e_1+\dots+x_ne_n)$ and,
vice versa, any homogeneous polynomial of degree $d$ in $n$ variables over $k$ is a form of degree $d$ and dimension $n$
over $k$. Any $d$-linear form $\theta : V \times
\dots \times V \to k$ induces a form $\varphi: V\to k$ of degree $d$ via
$\varphi (v)=\theta(v,\dots,v)$.
 We can hence identify $d$-linear forms and forms of degree $d$.

Two $d$-linear spaces $(V_i,\theta_i)$, $i=1,2$, are {\it isomorphic} (written
$(V_1,\theta_1)\cong (V_2,\theta_2)$ or just $\theta_1\cong\theta_2$) if there exists a bijective linear map
$f:V_1\to V_2$ such that $\theta_2(f(v_1),\dots,f(v_d))=\theta_1(v_1,\dots,v_d)$ for all $v_1,\dots,v_d\in V_1.$
A $d$-linear space $(V,\theta)$ (or the $d$-linear form $\theta$) is
 {\it nondegenerate} if $v = 0$ is the only vector such
that $\theta (v, v_{2}, \dots, v_d) = 0$ for all $v_i \in V$. A form of degree $d$ is
called {\it nondegenerate} if its associated $d$-linear form
is nondegenerate. A form $\varphi$ over $k$ is called {\it anisotropic}, if it does not have any non-trivial zeroes,
otherwise it is called {\it isotropic}.

The {\it orthogonal sum} $(V_1,\theta_1)\perp (V_2,\theta_2)$ of two $d$-linear spaces $(V_i,\theta_i)$, $i=1,2$, is
 the $k$-vector space $V_1\oplus V_2$ together with the $d$-linear form
$$(\theta_1 \perp\theta_2)(u_1+v_1,\dots,u_d+v_d)=\theta_1(u_1,\dots,u_d)+\theta_2(v_1,\dots,v_d)$$
($u_i\in V_1$, $v_i\in V_2$) \cite{H-P}.

A $d$-linear space $(V,\theta)$ is
called {\it decomposable} if $(V,\theta)\cong (V,\theta_1)\perp (V,\theta_2)$
for two non-zero $d$-linear spaces $(V,\theta_i)$, $i=1,2$.
 If $\varphi$ is represented by the homogeneous polynomial $a_1x_1^d+\ldots +a_mx_m^d$ ($a_i\in k^\times$) we use the notation $\varphi=\langle  a_1,\ldots
,a_n\rangle  $ and call $\varphi$ {\it diagonal}.
A diagonal form $\varphi=\langle  a_1,\ldots,a_n\rangle  $ over $k$ is nondegenerate if and only if $a_i\in k^\times$
for all $1\leq i\leq n$.

If $d\geq 3$, $a_i,b_j\in k^\times$, then
$\langle  a_1,\ldots ,a_n\rangle\cong \langle  b_1,\ldots ,b_n\rangle$
 if and only if there is a permutation $\pi\in S_n$ such that
 $\langle b_i \rangle\cong \langle a_{\pi (i)}\rangle$ for every $i$.
This is a special case of \cite[Theorem 2.3]{H}.

Note that for quadratic forms ($d=2$), the automorphism group of $\varphi$ is infinite, whereas
for $d\geq3$, the automorphism group of  $\varphi$ usually is finite,  for instance if $\varphi$ is
is nonsingular in the sense of algebraic geometry \cite{S}.
In particular, nondegenerate diagonal forms of degree $d\geq 3$ have finite automorphism groups  \cite[p.~137]{H},
which creates a problem when trying to imitate patching arguments as it is not possible
to apply \cite[Theorem 3.7]{H-H-K}.

\subsection{Higher degree $u$-invariants}

The {\it
$u$-invariant (of degree $d$)} of $k$ is defined as $u(d, k) =
\sup \{ \dim_k \varphi \}$, where $\varphi$ ranges over all the anisotropic forms
of degree $d$ over $k$.
The {\it diagonal $u$-invariant (of degree $d$)} of $k$  is defined as $u_{diag}(d, k) = \sup \{ \dim \varphi \}$, where
$\varphi$ ranges over all the anisotropic diagonal forms over $k$.

Thus the diagonal $u$-invariant
$u_{diag}(d, k)$ is the smallest integer $n$ such that all diagonal forms of degree $d$ over
$k$ of dimension greater than $n$ are isotropic,
 and the $u$-invariant $u(d, k)$ is the smallest integer $n$ such that all forms of degree $d$ over
$k$ of dimension greater than $n$ are isotropic. Obviously, $u_{diag}(d, k) \leq u (d, k).$
 If $u = u(d, k)$ then each anisotropic form of degree $d$ over $k$ of
dimension $u$ is universal. If $u = u_{diag}(d, k)$ then each diagonal anisotropic form of degree $d$ over $k$ of
dimension $u$ is universal. We have
 $$u_{diag}(d, k) \leq \min \{ n \, \vert \,{\rm all}\,{\rm forms}\,{\rm of}\,{\rm degree}\, d \,{\rm over}\, k \, {\rm of}\,
  {\rm dimension}\, \geq n \,{\rm are}\,{\rm universal}\}$$
 with the understanding that the ``minimum'' of an empty set of integers
  is  $\infty$, cf. \cite{Pu}. For $d=2$, $u_{diag}(d, k)= u (d, k)$ is the $u$-invariant of quadratic forms.

For an algebraically closed field $k$, $\vert k^{\times} /k^{\times d}\vert = 1$ and hence
$u_{diag}(d, k)=u(d, k) = 1.$
For a formally real field $k$, the diagonal $u$-invariant is infinite  for even $d$:
since $-1\not\in \sum k^2$, also $-1\not\in \sum k^d$ for any  even $d$. Thus the form $m\times \langle  1\rangle
$ of degree $d$
 is anisotropic  for each integer $m$, implying
$u_{diag}(d,k)=u(d,k)=\infty.$

 The {\it strong diagonal} $u$-invariant of degree $d$ of $k$, denoted $u_{diag,\, s}(d, k)$, is the smallest real
number $n$ such that
\begin{enumerate}
\item every finite field extension $E/k$ satisfies $u_{diag}(d, E)\leq n$, and
\item every finitely generated field extension $E/k$ of transcendence degree one satisfies $u_{diag}(d, E)\leq dn$.
\end{enumerate}
If these $u$-invariants are arbitrarily large, put $u_{diag,\, s}(d, k)=\infty.$

Analogously as observed in \cite{H-H-K} for $d=2$,  $u_{diag,\, s}(d, k)\leq n$ if and only if
every finitely generated field extension $E/k$ of transcendence degree $l\geq 1$ satisfies $u_{diag}(d, k)\leq d^ln$.
Thus if $u_{diag,\, s}(d, k)$ is finite, it is at least 1 and lies in $\frac{1}{d}\mathbb{N}$.

\subsection{$C_r^0$ fields}
 Let $r\geq 1$ be an integer.
A field $F$ is a \emph{$C_r$-field} if for all $d \ge 1$ and $n > d^r$, every homogeneous form of
degree $d$ in $n$ variables over $F$ has a non-trivial solution in $F$.  In particular,
then $F$ satisfies $u(d,F) \le d^r$.  Moreover, every finite extension of $F$ is  a $C_r$-field,
and every one-variable function field over $F$  a $C_{r+1}$-field \cite[II.4.5]{Ser}. Hence
  $u_{diag,\, s}(d,F) \le d^r$  for a $C_r$-field $F$.

  A field $F$ is a \emph{$C_r^0$-field} if the following  holds:
For any finite field extension $F'$ of $F$ and any integers $d\geq 1$ and $n > d^r$, for any homogeneous
form over $F'$ of degree $d$ in $n$ variables, the greatest common divisor of the degrees of finite field extensions $F''/F'$
over which the form acquires a nontrivial zero is one.
This amounts to requiring that the $F'$-hypersurface defined by the form has a
zero-cycle of degree 1 over $F'$.

Assume ${\rm char}(F) = 0$. For each prime $l$, let $F_l$ be the fixed field of a pro-$l$-Sylow subgroup of
the absolute Galois group of $F$. Any finite subextension of $F_l/F$ is of degree coprime to $l$.
The field $F$ is $C_r^0$ if and only if each of the fields $F_l$ is $C_r$.
A finite field extension of a $C_r^0$-field is $C_r^0$.
If $F$ is $C_r^0$ then a function field $E = F(x_1,\dots,x_s)$
in $s$ indeterminates $x_1,\dots,x_s$ over $F$ is $C_{r+s}^0$  \cite[2.1]{CT-P-S}.

It is not known if
$p$-adic fields have the $C^0_2$-property.

\begin{remark}
Assume that $p$-adic fields have the $C^0_2$-property.
Let $K(X)$ be any function field of transcendence degree $r$ over a $p$-adic field $K$ (here we do not need to assume $p \not=2, 3$).
 Suppose that there is $\ell\not=2$ such that there exists a finite
subextension of $K_\ell(X)/K(X)$  of degree 2.
Then any cubic form over $K(X)$ in strictly more than
$3^{2+r}$ variables has a nontrivial zero:
If the $p$-adic field $K$ is $C_3^0$, then the function field $E=K(X)$ in $r$ variables over $K$ is $C_{3+r}^0$
 \cite[Lemma 2.1]{CT-P-S}.
Thus a cubic form over $E=K(X)$ in strictly more than $3^{2+r}$ has a nontrivial zero in each
of the fields $K_\ell(X)$, $l$ a prime, hence in a finite extension of $K(X)$ of degree coprime to $\ell$, for each $\ell$ prime.
Pick $\ell\not=2$, then $[K_\ell(X):K(X)]$ is even.
Moreover, pick $l\not=2$ such that there exists a finite
subextension of $E_\ell/K(X)$  of degree 2
then the cubic form has a zero over it.
 By Springer's Theorem for cubic forms and their behaviour under quadratic field extensions
 \cite[VII]{La}, thus the cubic form has a nontrivial zero in $K(X)$.
 This is the analogue of \cite[Proposition 2.2]{CT-P-S}.
\end{remark}

\section{Diagonal forms over Henselian valued fields}\label{sec:valued}

\subsection{} Let $K$ be a  valued field with valuation $v$, valuation ring  $R$ and maximal ideal
$m$.  Let $\Gamma$ be the value group. Assume that $d!$ is not divisible by the characteristic of the residue field $k=R/m$.
For $u\in R$, denote by $\bar u$ the image of $u$ in $k $.
For a polynomial $f\in R[X]$, $f=a_nx^n+\dots+a_1x+a_0$, define the polynomial
$\overline{f}= \bar a_nx^n+\dots+ \bar a_1x+\bar a_0$ over $k$.
If $\varphi=\langle a_1,\dots,a_n\rangle$ is a nondegenerate diagonal form with entries $a_i\in R$, define the
diagonal form
$\overline{\varphi}=\langle \bar a_1,\dots,\bar a_n\rangle$ over $k$.
$\varphi$  is called a {\it unit form}, if $\overline{\varphi}$ is nondegenerate.
Choose a set $\{ \pi_{\gamma}\in R \; \vert
\; \gamma \in I \}$ such that the values of the $ \pi_{\gamma}$'s  represent the distinct cosets in $\Gamma
/d \Gamma$. We may decompose a diagonal form $\varphi$ as $\varphi=\perp\varphi'_\gamma$ by taking $\varphi'_\gamma$ to be the
diagonal form whose entries comprise all $a_i$ with $v(a_i)=v(\pi_\gamma)\, {\rm mod}\, d\,\Gamma$.
By altering the slots by $d$-powers if necessary, we may then write $\varphi'_\gamma=\pi_\gamma\varphi_\gamma$
with each $\varphi_\gamma$ a diagonal unit form.  There are only finitely many non-trivial $\varphi_\gamma$ \cite{Mo}.
If $\Gamma=\mathbb{Z}$, the set $\{ \pi_{\gamma} \; \vert\; \gamma \in I \}$ can be chosen to be
$\{ \pi^{i} \; \vert\; i=0,\dots,d-1 \}$ and $|\Gamma/d\Gamma|=d$ is finite.

If $R$ satisfies Hensel's Lemma then $(K, v)$ is called a {\it Henselian valued field} and
$R$ a {\it Henselian valuation ring}. Every complete discretely valued field is Henselian.

Let $\varphi$ be a diagonal form over a Henselian valued field $(K, v)$. Write $\varphi=\pi_1\varphi_1\perp
\dots\perp \pi_r\varphi_r$ with each $\varphi_i$ a diagonal unit form and the $\pi_i$ having distinct values in
$\Gamma/d \Gamma$. Then $\varphi$ is isotropic if and only if some $\overline{\varphi_i}$ is isotropic \cite[Proposition 3.1]{Mo}.
This is because for a diagonal unit form $\varphi$ over a Henselian valued field  $(K, v)$,
 $\varphi$ is isotropic if and only if $\overline{\varphi}$ is isotropic  \cite[Lemma 2.3]{Mo}.

\begin{theorem} (\cite{Mo} or \cite[Theorem 4, Corollary 2]{Pu}) \label{thm:4}
Suppose that  ${\rm char} (k) \nmid d!$.
\\ (i) Let $(K, v)$ be a  Henselian  valued field. Then
$u_{diag}(d, K)= |\Gamma /d \Gamma| \,u_{diag}(d,k).$
\\ (ii)
 Let $(K, v)$ be a Henselian valued field.
 If every diagonal form of degree $d$ of dimension $n+1$ over $k$ is isotropic, then every diagonal form of degree $d$
 and dimension $dn+1$ over $K$ is isotropic. If $k$ has an anisotropic form of degree $d$ and dimension $n$, then $K$
 has an anisotropic form of degree $d$ and dimension $dn$.
\\ (iii) Let $K$ be a discretely valued field. Then
 $u_{diag}(d, K)\geq d \, u_{diag}(d, k).$
 \\ (iv)
Let $F$ be a field extension of finite type over $k$ of transcendence degree $n$. Then
$u_{diag}(d, F)\geq d^n u_{diag}(d, k^{\prime})$
for a suitable finite field extension $k^{\prime}/k$.
\\ The  (in)equalities in (i), (iii), (iv) also hold when the values are infinite.
\end{theorem}

For $d=2$, (ii) is Springer's Theorem for quadratic forms over Henselian valued fields \cite{Sp}.
 Springer's Theorem does not hold for non-diagonal forms of higher degree than 2 \cite[2.7]{Mo}.
 Theorem \ref{thm:4} is a major ingredient in our proofs, for instance we can show:

\begin{proposition} \label{le:7}
Let $A$ be a discrete valuation ring with fraction field $K$ and residue field $k$
such that ${\rm char} (k) \nmid d!$.
\\ (i) $u_{diag}(d,K)\geq d \, u_{diag}(d,k)$ and $u_{diag,\, s}(d,K)\geq d\, u_{diag,\, s}(d,k)$.
\\ (ii) If $A$ is Henselian then $u_{diag}(d,K)= d \, u_{diag}(d,k)$.
\\ (iii) If $A$ is a complete discrete valuation ring then every finite extension of $K$ has diagonal $u$-invariant at most
$d\,u_{diag}(d,k)$.
\end{proposition}

The first assertions of (i) as well as (ii) and (iii) follow from Theorem \ref{thm:4}.
The proof of the second claim in (i) is analogous to the one of \cite[4.9]{H-H-K}, employing Theorem
\ref{thm:4} instead of Springer's Theorem.

\subsection{}

A field $K$ is called an \textit{$m$-local field} with
\textit{residue field $k$} if there is a sequence of fields $k_0,\dots,k_m$
with $k_0=k$ and $k_m=K$, and such that $k_i$ is the fraction field of an
excellent Henselian discrete valuation ring with residue field $k_{i-1}$ for
$i=1,\dots,m$.  Recall that a discrete valuation ring $R$ is called \textit{excellent}, if the field extension ${\wh K}/K$
is separable, where ${\wh K}$ denotes the quotient field of $R$
and $K$ is its completion. (This condition is trivially satisfied if $K$ has characteristic 0 or
$R$ is complete.)

Proposition \ref{le:7} implies (compare the next two results with \cite[Corollary 4.13, 4.14]{H-H-K} for quadratic forms):

\begin{corollary} \label{cor:old9}
Suppose that $K$ is an $m$-local field whose residue field $k$ is a $C_r$-field with ${\rm char} (k) \nmid d!$.
Let $F$ be a function field over $K$ in one variable.
\\ (i) $u_{diag}(d,k)=u_{diag,s}(d,k)=d^{r}$ and
$u_{diag}(d,K) =  d^{r+m}$.
\\ Moreover, if some normal $K$-curve with function field $F$
has a $K$-point, then\\
 $u_{diag}(d,F) \ge d^{r+m+1}$.
\\ (ii)
If $u_{diag}(d,k')=d^r$ for every finite extension $k'/k$, then $u_{diag}(d,F) \ge d^{r+m+1}$.
\end{corollary}

\begin{proof}
(i) Since $k$ is a $C_r$-field, $u_{diag}(d,k)=d^r$, thus
$u_{diag}(d,k) \le u_{diag, s}(d,k) \le d^r$ and the first two equations follow.
Applying Proposition \ref{le:7}
and induction yields that $u_{diag}(d,K)
\ge d^m u_{diag}(d,k) = d^{r+m}$.
Let $X$ be a normal $K$-curve with function field $F$ and let $\xi$ be a $K$-point on $X$.
The local ring at $\xi$ has fraction field $F$ and residue field $K$.  So Proposition \ref{le:7}
implies that $u_{diag}(d,K) =  d^{m}u_{diag}(d,k) $ and
$u_{diag}(d,F) \ge d u_{diag}(d,K) =d^{r+m+1}$.
\\ (ii) Choose a normal or equivalently a regular $K$-curve $X$ with function field $F$, and a closed point
 $\xi$ on $X$.  Let $R$ be the local ring of $X$ at $\xi$ with residue field $\kappa(\xi)$.
 Then the fraction field of $R$ is $F$, and $\kappa(\xi)$ is a finite extension of $K$. Hence $\kappa(\xi)$
 is an $m$-local field whose residue field $k'$ is a finite extension of $k$. By assumption, $u_{diag}(d,k')=d^r$
and $k'$ is a $C_r$-field since $k$ is. So applying part (ii)
to $k'$ and $\kappa(\xi)$, it follows that $u_{diag}(d,\kappa(\xi)) \ge  d^{r+m}$. Proposition \ref{le:7}
 yields $u_{diag}(d,F) \ge d^{r+m+1}$.
\end{proof}

\begin{corollary} \label{cor:old10}
(i)
Let $F$ be a one-variable function field over an $m$-local field $K$ with residue field $k$ such that
${\rm char} (k) \nmid d!$ and
$k$ is algebraically closed. Then $u_{diag}(d,F) \ge d^{m+1}$.
\\ (ii) If $k$ is a finite field and $u_{diag}(d,k)=d^r$ with $r\in\{0,1\}$, then
$u_{diag}(d,k)=u_{diag,s}(d,k)=d^{r}$ and $u_{diag}(d,K) \ge  d^{r+m}$.
 Moreover, if some normal $K$-curve with function field $F$
has a $K$-point, then $u_{diag}(d,F) \ge d^{r+m+1}$.
\end{corollary}

\begin{proof}
(i) This is a special case of Corollary \ref{cor:old9}
 using that an algebraically closed field $k$ is $C_0$, satisfies $u_{diag}(d,k)=1$,
 and has no non-trivial finite extensions.\\
 (ii) A finite field is $C_1$  \cite[II.3.3(a)]{Ser}, hence $u_{diag}(d,k)\leq d$.
 Here $u_{diag}(d,k)=u_{diag,\,s}(d,k)=d^r$ and Corollary \ref{cor:old9}
 yields the assumption.
\end{proof}

In general,  for any finite field $k={\mathbb F}_q$ we obviously do not have $u_{diag}(d,k)=d^r$, $r\in\{0,1\}$: for instance, if
$-1\in {\mathbb F}_q^{\times d}$ and $d \geq 4$ then
$u_{diag}(d,{\mathbb F}_q)\leq d-1$
by \cite{O}. Or, if $d^*={\rm gcd}(d,q-1)$, then
$u_{diag}(d,\mathbb{F}_q) \leq  d^*.$
 This implies that $u_{diag}(d,\mathbb{F}_q)=1$, if $d$ is relatively prime to $q-1$ and that
for  $q > (d^*-1)^4$, $u_{diag}(d, \mathbb{F}_q)=2$ \cite[5.1]{Pu}.

%
%
\section{The behaviour of diagonal forms of higher degree over function fields of $p$-adic curves} \label{sec:Parimala}

Whenever we write `discrete valuation ring' and `discrete valuation' we mean a discrete valuation
ring of rank one and a valuation with value group $\mathbb{Z}$.

\subsection{}
Let $A$ be a complete discrete valuation ring with fraction field $K$ and residue field
$k$ with ${\rm char} (k) \nmid d!$. Let $X$ be a smooth, projective, geometrically integral curve
over $K$ and $F = K(X)$ be the function field of $X$.
Let $t$ denote a uniformizing parameter for $A$.
 For each (rank one) discrete valuation $v$ of $F$, let $F_v$ denote the completion of $F$ with respect to $v$.

 We will adapt some ideas from \cite{CT-P-S} to diagonal forms of higher degree:
take a nondegenerate form $\varphi=\langle a_1,\dots,a_n\rangle$  of degree $d$ over $F$.
Then choose a regular proper model ${\mathcal X}/A$ of $X/K$, such that there exists a reduced divisor
$D$ with strict normal crossings which contains both the support of the divisor of all the entries $a_i$,
$1\leq i\leq n$,
and the components of the special fibre of $X/A$.
(Note that this implies that the regular proper model ${\mathcal X}/A$ depends on the form
$\varphi$, and thus so do  $Y,Y_i$, $S_0$, $S$, $F_P,F_U,\dots$ as defined in the following.)

 Let $Y = {\mathcal X}\times_Ak$ be the special fibre of $X/A$.
Let $x_i$ be the generic point of an irreducible component $Y_i$ of $Y$. Then there is an affine Zariski neighbourhood
$W_i\subset {\mathcal X}$ of $x_i$, such that the restriction of $Y_i$ to $W_i$ is a principal divisor.
Let $S_0$ be a finite set of closed points of $Y$ containing all singular points of
 $D$, and all the points that lie on some $Y_i$, but not in the corresponding $W_i$.

Choose a finite $A$-morphism $f : {\mathcal X}\to\mathbb{P}_A^1$
 as in \cite[Proposition 6.6]{H-H}. Let $S$ be the inverse image under $f$ of the point
at infinity of the special fibre $\mathbb{P}_k^1$. Then the set $S_0$ is contained in $S$. All the intersection points of two components
$Y_i$ are in $S$. Each component $Y_i$ contains at least one point of $S$.
Let $U \subset Y$ run through the reduced irreducible components of the complement of $S$ in $Y$.
Then each $U$ is a regular affine irreducible curve over $k$ and we define $k[U]$ to be its ring of regular functions
and $k(U)$ to be its fraction field. $k[U]$ is a Dedekind domain and $U ={\rm Spec}\, k[U]$.
Each $U$ is contained in an open affine subscheme ${\rm Spec}\,R^U$ of ${\mathcal X}$ and is a
principal effective divisor in ${\rm Spec}\,R^U$.
Moreover, $R_U$ is the ring of elements in $F$ which are regular on $U$ and also a regular ring,
since it is the direct limit of regular rings. The ring $R_U$ is a localisation of $R^U$ and so $U$ is a principal
effective divisor on ${\rm Spec}\,R_U$ given by the vanishing of an element $s \in R_U$.
The $t$-adic completion ${\wh R}_U$ of $R_U$ is a domain and  coincides with the $s$-adic
completion of $R_U$, since $t = us^r$
 for some integer $r \geq 1$ and a unit $u \in R_U^\times$.
 By definition, $F_U$ is the field of fractions of ${\wh R}_U$. We have
$k[U] = R_U/s ={\wh R}_U /s$.
For $P \in S$, the completion ${\wh R}_P$ of the local ring $R_P$ of $\mathcal{X}$ at $P$
is a domain and $F_P$ is the field of fractions of ${\wh R}_P$.
Let $p = (U, P)$ be a pair with $P \in S$ in the closure of an irreducible component $U$ of the
complement of $S$ in $Y$. Then let $R_p$ be the complete discrete valuation ring which is the completion
of the localisation of ${\wh R}_P$ at the height one prime ideal corresponding to $U$. Then
 $F_p$ is the field of fractions of ${\wh R}_p$ and $F$ is the inverse limit of the finite inverse system of fields
$\{F_U , F_P , F_p\}$ by \cite[Proposition 6.3]{H-H}.

The following can be seen as a weak generalization of \cite[Theorem 3.1]{CT-P-S} to diagonal forms of higher degree. Here we are not able to conclude that under the given assumptions, $\varphi$ is isotropic over $F$, only over
the $F_U$'s and $F_P$'s:

\begin{theorem} \label{thm:21}
Let $\varphi$ be a nondegenerate diagonal form of degree $d$ over
$F$. If $\varphi $ is isotropic over the
completion $F_v$  of $F$ with respect to each discrete valuation $v$ of $F$ with residue field either a function field in
one variable over $k$ or a finite extension of $K$,
then:
\\ (i) $\varphi$ is isotropic over $F_U$ for each reduced irreducible component $U\subset Y$ of the
complement of $S$ in $Y$,
\\ (ii) $\varphi$ is isotropic over  $F_P$ for each $P\in S$.
\end{theorem}

\begin{proof}
 Suppose  $\varphi=\left< a_1,\dots, a_n \right>$.
\\ (i) Each  entry $a_i$ of  $\varphi$ is supported only along $U$ in ${\rm Spec}\,R_U$,
thus has the form
$u s^j$ where $u\in R_U^\times$. We sort the entries $a_i=u_i s^j$ by the power $j$ of $s$ and use them to define  new
diagonal forms $\rho_j$ which  have all the $u_i$'s belonging to those $a_i$ where $s$ occurred
in the $j$th power as their diagonal entries. Hence  $\varphi$ is isomorphic to the diagonal form
$$\rho_0\perp s\rho_1\perp  \dots\perp s^{d-1}\rho_{d-1}$$
over $F$, where the $\rho_i$ are nondegenerate diagonal forms of degree $d$ over $R_U$.
Note that if for some $j\in\{ 0,1,\dots,d-1\}$
there is no $a_i$ with $a_i=u_i s^j$, then there is no corresponding form $\rho_j$ and a  $\rho_j$
 does not appear as a component in the sum.

By hypothesis, $\varphi$ is isotropic over the field of fractions of the completed local ring of
$\mathcal{X}$ at the generic point of $U$. By Theorem  \ref{thm:4}, this implies
that the image
of at least one of the forms $\rho_0$, $\rho_1$ or $\rho_{d-1}$ under the composite
homomorphism $R_U \to k[U] \to k(U)$ is isotropic over $k(U)$.
Since the residue characteristic $p$ does not divide $d!$,  the forms $\rho_0, \rho_1,\dots,\rho_{d-1}$
 define a smooth projective variety
over $R_U$. In particular, all of them define a smooth variety over $k[U]$. Since
$k[U]$ is a Dedekind domain, if such a projective variety has a point over $k(U)$, it has a point over $k[U]$. Since the
variety is smooth over $R_U$, a $k[U]$-point lifts to an ${\wh R}_U$-point (cf. the
 discussion after \cite[Lemma 4.5]{H-H-K}). Thus $\varphi$ has a nontrivial zero over $F_U$.
\\ (ii) Let $P \in S$. The local ring $R_P$ of $\mathcal{X}$ at $P$ is
regular. Its maximal ideal is generated by two elements $(x, y)$ with the property that any $a_i$ is
the product of a unit, a power of $x$ and a power of $y$. Thus over $F$, the fraction field of $R_P$,
  $\varphi$ is isomorphic to
$$\varphi_1\perp x\varphi_2\perp y\varphi_3\perp xy\varphi_4\perp x^2\varphi_5\perp y^2\varphi_6\perp x^2y^2\varphi_7
\perp x^2y\varphi_8\perp x y^2\varphi_9
\perp \dots \perp x^{d-1}y^{d-1}\varphi_{d^2},$$
where each $\varphi_i$ is a nondegenerate diagonal form over $R_P$. Let $R_y$ be the localization
of $R_P$ at the prime ideal $(y)$. $R_y$ is a discrete valuation ring with fraction field $F$. The
residue field $E$ of $R_y$ is the field of fractions of the discrete valuation ring $R_P /(y)$. By hypothesis, the
form
$$(\varphi_1\perp x\varphi_2\perp  x^2\varphi_5\perp\dots)\perp y(\varphi_3\perp x\varphi_4\perp\dots)
\perp y^2(\varphi_6\perp\dots )\perp\dots\perp  y^{d-1}(\dots\perp x^{d-1}\varphi_{d^2})$$
is isotropic over the field of fractions of the completion of
$R_y$. By Theorem  \ref{thm:4},  the reduction of one of the forms
$$(\varphi_1\perp x\varphi_2\perp  x^2\varphi_5\perp\dots),(\varphi_3\perp x\varphi_4\perp\dots ),\cdots,$$
 is isotropic  over $E$. Since $x$ is a uniformizing parameter for $R_P /(y)$, by
Theorem \ref{thm:4} this  implies that over the residue field
of $R_P /(y)$, the reduction of one of the forms $\varphi_1,\varphi_2,\varphi_3\dots ,\varphi_{{d}^2}$
 is isotropic. But then one of these forms
is isotropic over ${\wh R}_P$, hence over the field $F_P$ which is the fraction field of ${\wh R}_P$.
\end{proof}

\begin{remark}
(i) In the proof of Theorem \ref{thm:21},  one of the forms
is isotropic over ${\wh R}_P$, and since $R_p$ is the complete discrete valuation ring which is the completion
of the localisation of ${\wh R}_P$ at the height one prime ideal corresponding to $U$ when $p = (U, P)$,
this form is also  isotropic over $R_p$ and therefore over the field of fractions $F_p$ of ${\wh R}_p$.
This implies that if $\varphi $ is isotropic over the
completion of $F$ with respect to each discrete valuation of $F$,
then $\varphi$ is isotropic over $F_U$ for each reduced irreducible component $U\subset Y$ of the
complement of $S$ in $Y$, over  $F_P$ for each $P\in S$ and over $R_p$ for each $p = (U, P)$.
Since $F$ is the inverse limit of the finite inverse system of fields
$\{F_U , F_P , F_p\}$, $\varphi $ is isotropic over all overfields used in the inverse limit.
\\ (ii)
The  discrete valuation rings used  in the above proof are the
local rings at a point of codimension 1 on a suitable regular proper model $\mathcal{X}$ of $X$
determined by the choice of $\varphi$
(analogously as noted in \cite[Remark 3.2]{CT-P-S}).
\end{remark}

 Given a nondegenerate diagonal form $\varphi$ of degree $d$ and dimension greater than two over $F$,
 it is not clear whether the isotropy of  $\varphi$ over $F_v$
for each $v$ (respectively, of  $\varphi$ over all $F_U$, $F_P$ and $F_p$) implies that $\varphi$ is isotropic (the
 fact that ${\rm dim}\,\varphi>2$ is necessary: it is easy to adjust the example in
 \cite[Appendix]{CT-P-S} to two-dimensional diagonal forms of even degree).

\begin{corollary}\label{cor:4}
Let $r\geq 1$ be an integer and $d\geq 3$. Assume that any diagonal
form in strictly more than $dr$ variables over any function field in one variable over $k$ is
isotropic. Then:
\\ (i) Any diagonal form of degree $d$ and dimension $>d^2r$  over the function field
$F = K(X)$ of a curve $X/K$ is isotropic over $F_v$, for every discrete valuation $v$
with residue field either a function field in one variable over $k$ or a finite extension of $K$.
\\ (ii) Any diagonal form of degree $d$ and dimension $>d^2r$  over the function field
$F = K(X)$ of a curve $X/K$
 is isotropic over $F_U$ for each reduced, irreducible component $U\subset Y$ of the
complement of $S$ in $Y$ and is isotropic over  $F_P$ for each $P\in S$.
\end{corollary}

Note that $Y$ and $S$  depend on $\varphi$.

\begin{proof} (i)
 Let $L$ be a finite field extension of $K$. This is a complete discretely valued field with residue
field a finite extension $\ell$ of $k$. The assumption made on diagonal forms of degree $d$ over functions fields in
one variable over $k$, in particular diagonal forms of degree $d$ over the field $\ell(t)$, and
Theorem \ref{thm:4} applied to  $\ell((t))$ show that any diagonal form of dimension
$>r$  over $\ell$ has a zero. A second application of Theorem \ref{thm:4} yields that any
diagonal form of degree $d$ of dimension $>dr$ over $L$ is isotropic.
Let $\varphi$ be a diagonal form of dimension $n$ over $F$ with $n > d^2r$.
 By the assumption and Theorem \ref{thm:4}, $\varphi$ is isotropic
over $F_v$ for every discrete valuation $v$
with residue field either a function field in one variable over $k$ or a finite extension of $K$.
\\ (ii) follows from Theorem \ref{thm:21}.
\end{proof}

This shows that trying to extend \cite[Corollary 3.4]{CT-P-S} from quadratic  to diagonal forms of higher degree results in
 a much weaker version.

 \subsection{}
 Let $K$ be a $p$-adic field with residue field $k$ such that ${\rm char} (k) \nmid d!$.

\begin{corollary}\label{cor:5}
Any diagonal form of degree $d$ and dimension $>d^3+1$  over
a function field in one variable $F=K(t)$ is
\\ (i) isotropic over $F_v$, for every discrete valuation $v$
with residue field either a function field in one variable over $k$ or a finite extension of $K$;
\\ (ii) isotropic over $F_U$ for each reduced, irreducible component $U\subset Y$ of the
complement of $S$ in $Y$ and  isotropic over  $F_P$ for each $P\in S$.
\end{corollary}

\begin{proof}
Every finite field $k$ is $C_1$ and so every  diagonal
form of degree $d$  and dimension  $>d^2$ over any function field in one variable over $k$ (which is $C_2$) is
isotropic. Assertion (i) is a direct consequence of Theorem \ref{thm:4} and (ii)
 follows from  Corollary \ref{cor:4} (ii).
\end{proof}

So if $\varphi$ is a diagonal form of degree $d$ in at least $d^3+1$ variables over
 $\mathbb{Q}(t)$ then $\varphi$  is isotropic over $(\mathbb{Q}_p(t))_U$ for any $p \nmid d!$,
  for each reduced, irreducible component $U\subset Y$ of the
complement of $S$ in $Y$, and isotropic over  $(\mathbb{Q}_p(t))_P$ for each $P\in S$.

\begin{remark} Let us compare Corollary \ref{cor:5} with the Ax-Kochen-Ersov Transfer Theorem \cite{Ax-K}:
 given a degree $d$,
for almost all primes $p$, a form of degree $d$ over $\mathbb{Q}_p$ of dimension greater than or equal to $d^2+1$
is isotropic \cite[(7.4)]{G}.
Moreover,
for any form $\varphi$ of degree $d\geq 2$ and dimension greater than
$d^{3}$  over $\mathbb{Q}(t)$,  for almost all
primes $p$ the form $\varphi$ is isotropic over  $\mathbb{Q}_p(t)$ (\cite{Z} for $d=2$, \cite{Pu} for $d\geq 3$).
The model-theoretic proofs of both results
do not allow for a more concrete observation on which primes exactly are included here,
 nor can they be extended to other base fields.
\end{remark}

Stronger upper bounds on $u_{diag}(d,\mathbb{F}_q(t))$ will yield stronger results on its dimension,
since we only used the upper bound in the well known inequality
$d\cdot u_{diag}(d,\mathbb{F}_q) \leq u_{diag}(d,\mathbb{F}_q(t)) \leq d^2$
to prove Corollary \ref{cor:5}, for instance we obtain:

\begin{corollary}\label{cor:6}
Assume that
$u_{diag}(d,k(t))=dr<d^2$ for some $r\in\{1,\dots, d-1\}$. Let $\varphi$ be a
 diagonal form of degree $d$ and dimension $>d^2r+1$ over a function field in one variable $F=K(t)$. Then:
\\ (i) $\varphi$ is isotropic over $F_v$, for every discrete valuation $v$
with residue field either a function field in one variable over $k$ or a finite extension of $K$;
\\ (ii) $\varphi$ is isotropic over $F_U$ for each reduced, irreducible component $U\subset Y$ of the
complement of $S$ in $Y$ and  over  $F_P$ for each $P\in S$.
\end{corollary}

It is well known  that $u_{diag}(d,K)\leq d^2$
for a $p$-adic field $K$ with residue field $k=\mathbb{F}_q$ \cite{D-L}. Indeed,
$u_{diag}(d, K)=d \, u_{diag}(d, \mathbb{F}_q)$ by Theorem \ref{thm:4},
assuming that ${\rm char}\, \mathbb{F}_q=p \nmid d!$  as before, which shows that clearly
 $u_{diag}(d,K)$ can be smaller than $d^2$. On the other hand, Artin's conjecture that $\mathbb{Q}_p$ is a $C_2$-field is false for instance for forms
of degree 4.


\smallskip
\noindent
\begin{thebibliography}{1}

\bibitem{Ax-K}  Ax, J. , Kochen, S., {\it Diophantine problems over local fields I.} Amer. J. Math. 87 (1965) 605-630.

\bibitem{CT-P-S} Colliot-Th\'el\`ene, J.-L., Parimala, R., Suresh, V., {\it Patching and local global principles
for homogeneous spaces over function fields of $p$-adic curves.} Comment. Math. Helv. 87 (2012) (4), 1011-1033.

\bibitem{D-L}  Davenport, H., Lewis, D. J., {\it Homogeneous additive equations}. Proc. Royal Soc. Ser. A, 272 (1963),
443-460.

\bibitem[G]{G} Greenberg, M.J., ``Lectures on forms in many variables'', W.A. Benjamin, Inc.,
New York, Amsterdam,
1969.


\bibitem{H-H} Harbater, D., Hartmann, J., {\it Patching over fields}. Israel J. Math. 176 (2010),
61-107.


\bibitem{H-H-K} Harbater, D., Hartmann, J., Krashen, D., {\it Applications of patching to quadratic forms
and central simple algebras}. Invent. Math. 178 (2) (2009), 231-263.

\bibitem{H} Harrison, D. K., {\it A Grothendieck ring of higher degree forms}. J. Algebra 35 (1975), 123-138.

\bibitem{H-P} Harrison, D. K., Pareigis, B., {\it Witt rings of
higher degree forms}. Comm. Alg. 16 (6) (1988), 1275-1313.

\bibitem{La} S. Lang, ``Algebra.'' Third Edition. Addison-Wesley Publ. Comp. 1997.

\bibitem{Mo} Morandi, P., {\it Springer's theorem for higher degree forms}. Math. Z. 256 (1) (2007), 221-228.

\bibitem{O} Orzech, M., {\it Forms of low degree in finite fields}. Bull.
Austral. Math. Soc. 30(1) (1984), 45-58.

\bibitem{Pa-S} Parimala, R., Suresh, V., {\it Isotropy of quadratic forms over function fields of p-adic curves}. Inst.
Hautes Etudes Sci. Publ. Math. 88 (1998), 129-150.

\bibitem{Pu} Pumpl\"un, S.,  {\it $u$-invariants for forms of higher degree}. Expo. Math. 27 (2009), 37-53.

\bibitem{S} Schneider, J. E., {\it Orthogonal groups of nonsingular forms of higher degree}. J. Alg. 27 (1973), 112-116.

\bibitem{Ser}  Serre, J.-P., ``Cohomologie Galoisienne''. Fourth edition.
Lecture Notes in Math. 5, Springer-Verlag, Berlin, Heidelberg and New York, 1973.


\bibitem{Sp} Springer, T. A., {\it Quadratic forms over a field with a
discrete valuation}. Indag. Math. 17 (1979), 33-39.

\bibitem{Z} Zahidi, K., \emph{On the u-invariant of p-adic function fields.} Comm. Alg. 33 (7) (2005) 2307-2314.

\end{thebibliography}
\end{document}